\documentclass[11pt]{article}

\usepackage{amsfonts,amsmath,amssymb,amsthm}
\usepackage{graphicx}
\usepackage{indentfirst}
\usepackage{mathtools}

\usepackage{tocloft}
\usepackage{bm}

\usepackage{tikz}
\usetikzlibrary{shapes,chains,scopes,positioning,backgrounds,fit,shadows,calc,arrows.meta}
\usepackage{caption}
\newtheorem{definition}{Definition}
\newtheorem{theorem}{Theorem}

\newtheorem{proposition}{Proposition}

\theoremstyle{plain}\newtheorem*{the 2}{Theorem 4}
\theoremstyle{plain}\newtheorem*{the 3}{Theorem 5}
\theoremstyle{plain}\newtheorem*{the 6}{Theorem 6}

\theoremstyle{definition}
\theoremstyle{definition}\newtheorem*{case 1.1}{Case 1.1}
\theoremstyle{definition}\newtheorem*{case 1.2}{Case 1.2}
\theoremstyle{definition}\newtheorem{cas}{Case}
\theoremstyle{definition}
\theoremstyle{definition}\newtheorem*{case 1}{Case 1}
\theoremstyle{definition}\newtheorem*{case 2}{Case 2}
\theoremstyle{definition}\newtheorem*{case 3}{Case 1}
\theoremstyle{definition}\newtheorem*{case 4}{Case 2}
\theoremstyle{definition}\newtheorem*{case 5}{Case 1}
\theoremstyle{definition}\newtheorem*{case 6}{Case 2}

\title{Degree conditions for disjoint path covers in digraphs}
\author{Ansong Ma$^{1}$,  Yuefang Sun$^{2}$
\\
$^{1}$ School of Mathematics and Statistics,
Ningbo University,\\
Zhejiang 315211,  China, mas0710@163.com\\
$^{2}$ Corresponding author. School of Mathematics and Statistics,\\
Ningbo University,
Zhejiang 315211, China, sunyuefang@nbu.edu.cn}
\date{}
\begin{document}
\begin{sloppypar}
	\maketitle
\begin{abstract}
In this paper, we study degree conditions for three types of disjoint directed path cover problems: many-to-many $k$-DDPC, one-to-many $k$-DDPC and one-to-one $k$-DDPC, which are intimately connected to other famous topics in graph theory, such as Hamiltonicity and $k$-linkage, and have a strong background of applications.

Firstly, we get two sharp minimum semi-degree sufficient conditions for the unpaired many-to-many $k$-DDPC problem and a sharp Ore-type degree condition for the paired many-to-many $2$-DDPC problem.
Secondly, we obtain a minimum semi-degree sufficient condition for the one-to-many $k$-DDPC problem on a digraph with order $n$, and show that the bound for the minimum semi-degree is sharp when $n+k$ is even and is sharp up to an additive constant 1 otherwise. Finally, we give a minimum semi-degree sufficient condition for the one-to-one $k$-DDPC problem on a digraph with order $n$, and show that the bound for the minimum semi-degree is sharp when $n+k$ is odd and is sharp up to an additive constant 1 otherwise.

\vspace{0.2cm}


\vspace{0.2cm}

\end{abstract}

\section{Introduction}

\subsection{Motivation, terminology and notation}

For terminology and notation not defined here, we refer to \cite{Bang-Jensen09}. 
Note that all digraphs considered in this paper have no parallel arcs or loops, and a path always means a directed path. A {\it biorientation} of a graph $G$ is a digraph which is obtained from $G$ by replacing each edge with two arcs of opposite directions.
A {\it complete digraph} $\overleftrightarrow{K}_{n}$ is a biorientation of a complete graph $K_{n}$. A {\it complete bipartite digraph} $\overleftrightarrow{K}_{a,b}$ is a biorientation of a complete bipartite graph $K_{a,b}$. We use $E_{n}$ to denote an empty digraph  with order $n$. For a set $S$, we use $|S|$ to denote the number of elements in $S$.

Let $D=(V(D), A(D))$ be a digraph. An {\it $x$-$y$ path} in $D$ is a path which is from $x$ to $y$ for two vertices $x, y \in V(D)$. 
 For two vertices $x$ and $y$ on a path $P$ of $D$ satisfying $x$ precedes $y$, let $xPy$ denote the subpath of $P$ from $x$ to $y$. A {\it Hamiltonian path} is a path containing all vertices of $D$. A digraph $D$ is said to be {\it Hamiltonian-connected} if $D$ has an $x$-$y$ Hamiltonian path for every choice of distinct vertices $x$, $y \in V(D)$. Let $s\in V(D)$ and let $T = \{t_{1}, t_{2}, \dots ,t_{k}\}\subseteq V(D) \setminus \{s\}$.  Let $S$ = $\{s_{1}, s_{2}, \dots , s_{k}\}$ and $T = \{t_{1}, t_{2}, \dots ,t_{k}\}$ be any two disjoint subsets of $V(D)$. A {\it directed $k$-linkage} for $S$ and $T$ in $D$ is a set of $k$ pairwise vertex disjoint paths $\{P_{1}, P_{2}, \dots, P_{k}\}$ such that $P_{i}$ is an $s_{i}-t_{i}$ path for each $i\in \{1, 2, \dots, k\}$. An {\it $(s, T)$-fan} is a collection of pairwise internally disjoint paths $\{P_{1}, P_{2}, \dots, P_{k}\}$  such that $P_{i}$ is an $s-t_{i}$ path for each $i\in \{1, 2, \dots, k\}$.


Let $S$ = $\{s_{1}, s_{2}, \dots , s_{k}\}$ and $T = \{t_{1}, t_{2}, \dots ,t_{k}\}$ be any two disjoint subsets of $V(D)$. We call $S$ and $T$ a {\it source set} and a {\it sink set}, respectively. A set of $k$ pairwise disjoint paths $\{P_{1}, P_{2}, \dots, P_{k}\}$ of $D$ is a {\it many-to-many $k$-disjoint directed path cover} ({\em many-to-many $k$-DDPC} for short) for $S$ and $T$, if $\bigcup^{k}_{i=1}V(P_{i}) = V(D)$ and each $P_{i}$ is a path from an element of $S$ to an element of $T$. A many-to-many $k$-DDPC is {\it paired} if $P_i$ is an $s_{i}-t_{i}$ path for each $i \in \{1, 2, \dots, k\}$, and it is {\it unpaired} if, for some permutation $\sigma$ on $\{1, 2, \dots, k\}$, $P_{i}$ is an $s_{i}-t_{\sigma(i)}$ path for each $i \in \{1, 2, \dots, k\}$. Note that a paired many-to-many $k$-DDPC can be seen as a special unpaired many-to-many $k$-DDPC. Also, observe that a paired many-to-many $k$-DDPC for $S$ and $T$ is a directed $k$-linkage for $S$ and $T$ covering all vertices of $D$ (see \cite{Bang-Jensen09, Chud-Scott-SeymourAM, Chud-Scott-Seymour, Ferrara, Gould2006, Kuhn2008, Liu-Rolek-Stephens-Ye-Yu, Meng-Rolek-Wang-Yu, Sun-Yeo, Thomas-Wollan, Zhou-Qi-Yan} for more information on the problem of $k$-linkage).

There are two variants of many-to-many $k$-DDPC: one-to-many $k$-DDPC and one-to-one $k$-DDPC. Let $S$=$\{s\}$ and $T$=$\{t_{1}, t_{2}, \dots, t_{k}\}$ be any two disjoint subsets of $V(D)$. A set of $k$ paths $\{P_{1}, P_{2}, \dots, P_{k}\}$ of $D$ is a {\it one-to-many $k$-disjoint directed path cover} ({\em one-to-many $k$-DDPC} for short) for $S$ and $T$, if $\bigcup^{k}_{i=1}V(P_{i}) = V(D)$, each $P_{i}$ is an $s-t_{i}$ path and $V(P_{i}) \cap V(P_{j}) = \{s\}$ for $i \neq j$. Similarly, let $S$=$\{s\}$ and $T$=$\{t\}$. A set of $k$ paths $\{P_{1}, P_{2}, \dots, P_{k}\}$ of $D$ is a {\it one-to-one $k$-disjoint directed path cover} ({\em one-to-one $k$-DDPC} for short) for $S$ and $T$, if $\bigcup^{k}_{i=1}V(P_{i}) = V(D)$, each $P_{i}$ is an $s-t$ path and $V(P_{i}) \cap V(P_{j}) = \{s, t\}$ for $i \neq j$. By definition, a one-to-many $k$-DDPC for $S$=$\{s\}$ and $T$=$\{t_{1}, t_{2}, \dots, t_{k}\}$ is an $(s,T)$-fan covering all vertices of $D$, and a one-to-one $1$-DDPC for $S=\{s\}$ and $T=\{t\}$ is a Hamiltonian path from $s$ to $t$.

We now introduce the definitions of  paired/unpaired many-to-many $k$-coverable digraphs and  one-to-many/one $k$-coverable digraphs.

\begin{definition}[{\bf Paired/unpaired many-to-many $k$-coverable digraphs}]\label{Def:cover-umtm} 
Let $D$ be a digraph of order $n \geq 2k$, where $k$ is a positive integer. If there is a paired (resp. unpaired) many-to-many $k$-DDPC in $D$ for any disjoint source set $S$=$\{s_{1}, s_{2}, \dots, s_{k}\}$ and sink set $T$=$\{t_{1}, t_{2}, \dots, t_{k}\}$, then $D$ is paired (resp. unpaired) many-to-many $k$-coverable.

\end{definition}



\begin{definition}[{\bf One-to-many/one $k$-coverable digraphs}]\label{Def:cover-otm}
Let $D$ be a digraph of order $n \geq k + 1$, where $k$ is a positive integer. If there is a one-to-many (resp. one-to-one) $k$-DDPC in $D$ for any disjoint source set $S$=$\{s\}$ and sink set $T$=$\{t_{1}, t_{2}, \dots, t_{k}\}$ (resp. $T = \{t\}$), then $D$ is  one-to-many (resp. one-to-one) $k$-coverable.

\end{definition}

\subsection{Main results}

Recall that disjoint path cover problems are intimately connected to other famous topics in graph theory, such as Hamiltonicity, $k$-linkage. In addition, they have applications in several areas such as software testing, database design and code optimization \cite{Ntafos}. Therefore, these problems have attracted much attention from researchers \cite{Cao18, Chen10, Khn08, Li20, Lim16, Lu21, MSZ23, Park16, Park20, Park21, Wang2019, Yu2018, Zhou}. In this paper, we will focus on the following three types of disjoint path cover problems in digraphs: many-to-many $k$-DDPC problem, one-to-many $k$-DDPC problem and one-to-one $k$-DDPC problem.






\subsubsection{Many-to-many $k$-DDPC problem}

For the unpaired many-to-many $k$-DDPC problem, we will get a sharp minimum semi-degree sufficient condition as follows.

\begin{theorem} \label{The:main1}
Let $D$ be a digraph of order $n \geq 3k$, where $k$ is a positive integer. If $\delta^{0}(D) \geq \lceil (n + k)/2 \rceil$, then $D$ is unpaired many-to-many $k$-coverable. Moreover, the bound for $\delta^{0}(D)$ is sharp.
\end{theorem}

If the condition ``$n \geq 3k$" is replaced by ``$n=2k$", then the bound ``$\lceil (n + k)/2 \rceil$" for $\delta^{0}(D)$ in Theorem~\ref{The:main1} can be reduced to ``$\lceil (n + k)/2 \rceil - 1$", which is still sharp. We also prove that, when $n=2k$, a complete bipartite regular digraph $D$ with order $n$  is unpaired many-to-many $k$-coverable.

\begin{theorem} \label{The:main2}
The following assertions hold:
\begin{description}
\item[(i)]~Let $D$ be a digraph of order $n = 2k$, where $k$ is a positive integer. If $\delta^{0}(D) \geq \lceil (n + k)/2 \rceil - 1$, then $D$ is unpaired many-to-many $k$-coverable.	Moreover, the bound for $\delta^{0}(D)$ is sharp.
\item[(ii)]~The digraph $\overleftrightarrow{K}_{m,m}$ $(m \geq 2)$ is unpaired many-to-many $m$-coverable, but it is not unpaired many-to-many $k$-coverable when $1 \leq k < m$.
\end{description}
\end{theorem}



For the paired many-to-many $k$-DDPC problem, K\"{u}hn, Osthus and Young \cite{Khn08} deduced a sharp minimum semi-degree sufficient condition.

\begin{theorem}\cite{Khn08} \label{The:mtm}
Let $D$ be a digraph of order $n \geq Ck^{9}$, where $C$ is a sufficiently large constant and $k \geq 2$. If $\delta^{0}(D) \geq \lceil n / 2\rceil + k - 1$, then $D$ is paired many-to-many $k$-coverable. Moreover, the bound for $\delta^{0}(D)$ is sharp.
\end{theorem}

In this paper, we will obtain a sharp Ore-type degree condition for the paired many-to-many $2$-DDPC problem.

\begin{theorem}\label{The:2t2}
Let $D$ be a digraph of order $n$. If $d^{+}_{D}(x) + d^{-}_{D}(y) \geq n + 2$ for each $xy\not\in A(D)$, then $D$ is paired many-to-many $2$-coverable. Moreover, the bound for $d^{+}_{D}(x) + d^{-}_{D}(y)$ is sharp.
\end{theorem}

\subsubsection{One-to-many $k$-DDPC problem}

For the one-to-many $k$-DDPC problem, Zhou \cite{Zhou} demonstrated the following minimum semi-degree condition. 

\begin{theorem} \cite{Zhou} \label{The:otm}
Let $D$ be a digraph of order $n \geq Ck^{9}$, where $C$ is a sufficiently large constant and $k \geq 2$. If $\delta^{0}(D) \geq \lceil (n + k + 1) / 2\rceil$, then $D$ is one-to-many $k$-coverable.
\end{theorem}

Ma, Sun and Zhang \cite{MSZ23} continued to study the one-to-many $k$-DDPC problem in semicomplete digraphs. In their result, they decreased the lower bounds of the order and the minimum semi-degree condition in this class of digraphs as follows.

\begin{theorem} \cite{MSZ23} \label{The:otms}
Let $D$ be a semicomplete digraph of order $n\geq (9k)^{5}$, where $k \geq 2$. If $\delta^{0}(D)\geq \lceil (n+k-1)/2\rceil$, then $D$ is one-to-many $k$-coverable.
\end{theorem}

In this paper, by Theorem~\ref{The:main1}, we will make improvements for Theorem~\ref{The:otm} by replacing ``$n \geq Ck^{9}$" with ``$n \geq 3k$" and replacing ``$\delta^{0}(D) \geq \lceil (n + k + 1) / 2\rceil$" with ``$\delta^{0}(D) \geq \lceil (n + k)/2 \rceil$", which
is sharp when $n+k$ is even and is sharp up to an additive constant 1 otherwise.
\begin{theorem} \label{The:main3}
Let $D$ be a digraph of order $n \geq 3k$, where $k \geq 2$. If $\delta^{0}(D) \geq \lceil (n + k)/2 \rceil$, then $D$ is one-to-many $k$-coverable. Moreover, the bound for $\delta^{0}(D)$ is sharp when $n+k$ is even and is sharp up to an additive constant 1 otherwise.
\end{theorem}

\subsubsection{One-to-one $k$-DDPC problem}

For the one-to-one $k$-DDPC problem, Cao, Zhang and Zhou \cite{Cao18} proved the following result. 

\begin{theorem} \cite{Cao18} \label{The:oto}
Let $D$ be a digraph of order $n \geq Ck^{9}$, where $C$ is a sufficiently large constant and $k \geq 2$. If $\delta^{0}(D) \geq \lceil (n + k + 1)/2 \rceil$, then $D$ is one-to-one $k$-coverable.
\end{theorem}

In this paper, we will improve  Theorem~\ref{The:oto} by replacing ``$n \geq Ck^{9}$" with ``$n \geq k+1$" and replacing ``$\delta^{0}(D) \geq \lceil (n + k + 1) / 2\rceil$" with ``$\delta^{0}(D) \geq \lceil (n + k-1)/2 \rceil$", which is sharp when $n+k$ is odd and is sharp up to an additive constant 1 otherwise.

\begin{theorem} \label{The:main4}
Let $D$ be a digraph of order $n \geq k + 1$, where $k \geq 2$. If $\delta^{0}(D) \geq \lceil (n + k - 1)/2 \rceil$, then $D$ is one-to-one $k$-coverable.	
Moreover, the bound for $\delta^{0}(D)$ is sharp when $n+k$ is odd and is sharp up to an additive constant 1 otherwise.
\end{theorem}

\section{Many-to-many $k$-DDPC problem}\label{sec:Proof1}

To prove Theorems~\ref{The:main1} and~\ref{The:2t2}, we need the following result.


\begin{theorem} \cite{MOL1976} \label{The:n+1}
Let $D$ be a digraph of order $n$. If $d^{+}_{D}(x) + d^{-}_{D}(y) \geq n + 1$ for each $xy\not\in A(D)$, then $D$ is Hamiltonian connected.
\end{theorem}



In Theorem~\ref{The:main1}, we will get a sharp minimum semi-degree sufficient condition for the unpaired many-to-many $k$-DDPC problem. The example for the sharpness of the bound will be given in Proposition~\ref{pro2}.

\vspace{2mm}

\noindent
{\bf Theorem~\ref{The:main1}.}
{\em  Let $D$ be a digraph of order $n \geq 3k$, where $k$ is a positive integer. If $\delta^{0}(D) \geq \lceil (n + k)/2 \rceil$, then $D$ is unpaired many-to-many $k$-coverable. Moreover, the bound for $\delta^{0}(D)$ is sharp.}	

\begin{proof}

We prove the result by induction on $k$. The base step of $k = 1$ holds by Theorem~\ref{The:n+1}. Now we assume that $k\geq 2$ and prove the inductive step. Let $S = \{s_{1}, s_{2}, \dots, s_{k}\}$ and $T = \{t_{1}, t_{2}, \dots, t_{k}\}$ be disjoint source set and sink set in $D$, respectively. Let $$S_{1} = S \setminus \{s_{k}\}, T_{1} = T \setminus \{t_{k}\}, X = V(D) \setminus\{s_{k}, t_{k}\}.$$ We construct a digraph $D'$ of order $n'(= n - 1)$ from $D[X]$ by adding one new vertex $r$ such that
$$N^{+}_{D'}(r) = N^{+}_{D}(s_{k}) \cap X, \ N^{-}_{D'}(r) = N^{-}_{D}(t_{k}) \cap X. $$ 
Observe that  $$d^{-}_{D'}(r) \geq d^{-}_{D}(t_{k}) - 1, d^{+}_{D'}(r) \geq d^{+}_{D}(s_{k}) - 1,$$ and 
for every $z \in X$, $$d^{-}_{D'}(z) \geq d^{-}_{D}(z) - 1, d^{+}_{D'}(z) \geq d^{+}_{D}(z) - 1.$$ 
Hence, 
$$\delta^{0}(D') \geq \lceil (n + k)/2 \rceil - 1 = \lceil (n + k - 2)/2 \rceil = \lceil (n' + k - 1)/2 \rceil.$$ 

By the induction hypothesis, $D'$ has an unpaired many-to-many $(k-1)$-DDPC, say $\mathbf{P}=\{P_{1}, \dots, P_{k-1}\}$, for $S_{1}$ and $T_{1}$, where each $P_{i}$ ($1 \leq i \leq k-1$) is an $s_{i}-t_{\sigma(i)}$ path for some permutation $\sigma$ on $\{1, 2, \dots, k-1\}$.
As $$n' = n - 1 \geq 3k - 1 \geq 2k + 1,$$ there exists a path $P_{i}$, say $P_{1}$, which contains $r$ as an inner vertex. Let $r^{+}$ (resp. $r^{-}$) be the successor (resp. predecessor) of $r$ on $P_{1}$. Let $$P^{\ast}_{1}:=s_{1}P_{1}r^{-}t_{k},  P_{k}:=s_{k}r^{+}P_{1}t_{\sigma(1)}.$$ Observe that $\{P^{\ast}_{1}, P_{2}, \dots , P_{k}\}$ is an unpaired many-to-many $k$-DDPC for $S$ and $T$, that is, $D$ is unpaired many-to-many $k$-coverable. \qedhere
\end{proof}

\vspace{2mm}

\begin{proposition}\label{pro2}
For every $k\geq 1$ and every $n \geq 3k + 1$, there exists a digraph $D$ on $n$ vertices with minimum semi-degree $\lceil (n+k) / 2 \rceil - 1$ which is not unpaired many-to-many $k$-coverable. 
\end{proposition}

\begin{proof}    

The proof is divided into the following two cases according to the parity of $n+k$. 
\begin{cas}
$n+k$ is even.
\end{cas}
Let $D$ be a digraph of order $n \geq 3k$ such that $$V(D)=A\cup B, |A\cap B|=k, D[A]\cong D[B]\cong \overleftrightarrow{K}_{(n+k)/2}.$$ Observe that $$\delta^{0}(D) =(n+k)/2 - 1=\lceil (n+k) / 2 \rceil - 1.$$
Let $$S=\{s_{i} \colon 1 \leq i \leq k\} \subseteq A \setminus B, T=A \cap B=\{t_{i} \colon 1 \leq i \leq k\}.$$ 
It can be checked that for any permutation $\sigma$ on $\{1, 2, \dots, k\}$, any set of disjoint paths $\{P_i \colon 1 \leq i \leq k\}$ does not cover vertices of $B\setminus A$, where $P_i$ is an $s_i-t_{\sigma(i)}$ path. Hence, there is no unpaired many-to-many $k$-DDPC for $S$ and $T$ in $D$, that is, $D$ is not unpaired many-to-many $k$-coverable. 
\begin{cas}
	$n+k$ is odd.
	\end{cas}
Let $D_{1}\cong \overleftrightarrow{K}_{(n+k-1)/2}$ and $D_{2}\cong E_{(n-k+1)/2}$.
Let $D$ be a new digraph of order $n \geq 3k+1$ such that 	
$V(D) = V(D_{1}) \cup V(D_{2})$ and 
$$A(D) = A(D_{1}) \cup A(D_{2}) \cup \{xy, yx \colon x \in V(D_{1}), \, y \in V(D_{2})\}.$$
Observe that $$\delta^{0}(D) =(n+k-1)/2 = \lceil (n+k) / 2 \rceil - 1$$ and $$|V(D_1)|=(n+k-1)/2\geq (3k+1+k-1)/2\geq 2k.$$
Let $S=\{s_{1}, s_{2}, \dots, s_{k}\}$ and $T=\{t_{1}, t_{2}, \dots, t_{k}\}$ be two disjoint subsets of $V(D_{1})$. It is not hard to check that for any permutation $\sigma$ on $\{1, 2, \dots, k\}$, a set of disjoint paths $\{P_i \colon 1 \leq i \leq k\}$ covers at most $$(n+k-1)/2-k=(n-k-1)/2(<|V(D_2)|)$$ vertices of $V(D_{2})$, where $P_i$ is an $s_i-t_{\sigma(i)}$ path. Hence, there is no unpaired many-to-many $k$-DDPC for $S$ and $T$ in $D$, that is, $D$ is not unpaired many-to-many $k$-coverable.  
\end{proof}


If the condition ``$n \geq 3k$" is replaced by ``$n=2k$", then the bound ``$\lceil (n + k)/2 \rceil$" for $\delta^{0}(D)$ in Theorem~\ref{The:main1} can be reduced to ``$\lceil (n + k)/2 \rceil - 1$", which is still sharp. The example for the sharpness of the bound will be given in Proposition~\ref{pro3}. We also prove that, when $n=2k$, a complete bipartite regular digraph $D$ with order $n$  is unpaired many-to-many $k$-coverable.

\vspace{2mm}

\noindent
{\bf Theorem~\ref{The:main2}.}
{\em The following assertions hold:
\begin{description}
\item[(i)]~Let $D$ be a digraph of order $n = 2k$, where $k$ is a positive integer. If $\delta^{0}(D) \geq \lceil (n + k)/2 \rceil - 1$, then $D$ is unpaired many-to-many $k$-coverable.	Moreover, the bound for $\delta^{0}(D)$ is sharp.
\item[(ii)]~The digraph $\overleftrightarrow{K}_{m,m}$ $(m \geq 2)$ is unpaired many-to-many $m$-coverable, but it is not unpaired many-to-many $k$-coverable when $1 \leq k < m$.
\end{description}}	

\begin{proof}
\noindent{\bf Part (i)}. When $n = 2k$, $\delta^{0}(D) \geq \lceil (n + k)/2 \rceil - 1 = \lceil 3k/2 \rceil - 1$. The base case that $k = 1$ is clear. Now we assume that $k \geq 2$. Let $S = \{s_{1}, s_{2}, \dots, s_{k}\}$  and $T = \{t_{1}, t_{2}, \dots, t_{k}\}$ be any two disjoint subsets of $V(D)$. 
Let $D_{1}=D[S]$ and $D_{2}=D[T]$. By deleting $A(D_{1})$, $A(D_{2})$ and the arcs from $D_{2}$ to $D_{1}$, we get a spanning subdigraph, say $D_{3}$, of $D$. Observe that $D_{3}$ is a bipartite digraph which only contains the arcs from $S$ to $T$ of $D$. 
Let $G$ be the underlying graph of $D_3$. 

We claim that $\lvert N_{G}(Y) \rvert \geq \lvert Y \rvert$ for all $Y \subseteq S$, where $N_{G}(Y)$ denotes the set of neighbours of vertices of $Y$ in $G$.
Suppose that there is a subset $Z \subseteq S$ such that $\lvert N_{G}(Z) \rvert < \lvert Z \rvert$. Note that for any $s_{i} \in S$, 
$$d_{G}(s_{i}) = d^{+}_{T}(s_{i}) \geq \lceil 3k/2 \rceil - 1 - (k - 1) = \lceil k/2 \rceil,$$ 
and thus $\lvert Z \rvert > \lvert N_{G}(Z) \rvert \geq \lceil k/2 \rceil$, which means that $\lvert Z \rvert \geq \lceil k/2 \rceil + 1$. For any $t_{j} \in T \setminus N_{G}(Z)$, we have $$d_{G}(t_{j}) = d^{-}_{S}(t_{j}) \leq k - \lvert Z \rvert \leq k - 
\lceil k/2 \rceil - 1 = \lfloor k/2\rfloor - 1$$ 
by definitions of $Z$ and $N_{G}(Z)$. Thus, we have 
$$d^{-}_{D}(t_{j}) = d^{-}_{S}(t_{j}) + d^{-}_{T}(t_{j}) \leq \lfloor k/2\rfloor - 1 + k-1 \leq \lfloor 3k/2\rfloor - 2,$$ 
which contradicts with the fact that $\delta^{0}(D) \geq \lceil 3k/2 \rceil - 1$. Thus, $\lvert N_{G}(Y) \rvert \geq \lvert Y \rvert$ for all $Y \subseteq S$, this means that $G$ contains a perfect matching $M$, from which we can obtain a perfect matching $M'$ of $D_3$, which is also an unpaired many-to-many $k$-DDPC  from $S$ to $T$ in $D$.  Hence, $D$ is unpaired many-to-many $k$-coverable.

\noindent{\bf Part (ii)}. Let $X$, $Y$ be the bipartition sets of $V(\overleftrightarrow{K}_{m,m})$. We first consider the case that $k=m$. Let $S$ and $T$ be a pair of disjoint source set and sink set, respectively, such that $|S|=|T|=m$. Let $$S_{X}=S\cap X, \ T_{X}=T\cap X, \ S_{Y}=S\cap Y, \ T_{Y}=T\cap Y.$$ Observe that $\lvert S_{X} \rvert = \lvert T_{Y} \rvert$ and $\lvert T_{X} \rvert = \lvert S_{Y} \rvert$. The subdigraph induced by $S_{X} \cup T_{Y}$ is a complete bipartite regular digraph which has a perfect matching $M_{1}$ such that each arc $e\in M_{1}$ is from $S_{X}$ to $T_{Y}$. Similarly, the subdigraph induced by $T_{X} \cup S_{Y}$ is a complete bipartite regular digraph which has a perfect matching $M_{2}$ such that each arc $e\in M_{2}$ is from $S_{Y}$ to $T_{X}$. Clearly, $M_{1} \cup M_{2}$ constitutes an unpaired many-to-many $m$-DDPC for $S$ and $T$. Hence, $D$ is unpaired many-to-many $m$-coverable.

 We next consider the case that $1 \leq  k < m$.  Suppose that $\overleftrightarrow{K}_{m,m}$ is unpaired many-to-many $k$-coverable, that is, there exists an unpaired many-to-many $k$-DDPC, say $\{P_i: 1\leq i\leq k\}$, for any disjoint source set $S=\{s_{1}, s_{2}, \dots, s_{k}\}$ and sink set $T=\{t_{1}, t_{2}, \dots, t_{k}\}$. We consider the case that  $(T\cup \{s_k\}) \subseteq Y$ and  $(S\setminus \{s_{k}\})\subseteq X$, and will get a contradiction in this case. Without loss of generality, we assume that $P_{k}$ starts at $s_{k}$. By the fact that $(T\cup \{s_k\}) \subseteq Y$, we have $|V(P_k)\cap X|+1=|V(P_k)\cap Y|$. For each $1 \leq i \leq k - 1$, we have $|V(P_i)\cap X|=|V(P_i)\cap Y|$. By the definition of an unpaired many-to-many $k$-DDPC,
 we have $$|X|+1=\sum_{i=1}^k{|V(P_i)\cap X|}+1=\sum_{i=1}^k{|V(P_i)\cap Y|}=|Y|,$$
which contradicts with the fact that $|X|=|Y|$. Therefore, $\overleftrightarrow{K}_{m,m}$ $(m \geq 2)$ is not unpaired many-to-many $k$-coverable when $1 \leq k < m$.
\qedhere
\end{proof}


\vspace{2mm}

\begin{proposition}\label{pro3}
    
For every $k\geq 2$ and every $n = 2k$, there exists a digraph $D$ on $n$ vertices with minimum semi-degree $\lceil (n+k) / 2 \rceil - 2$ which is not unpaired many-to-many $k$-coverable.
\end{proposition}

\begin{proof}
    
The proof is divided into the following two cases according to the parity of $k$.

\begin{case 5}
$k$ is odd.
\end{case 5}

Let $D$ be a digraph of order $n = 2k$ such that 
$$V(D)=A\cup B, |A\cap B|=k-1, D[A]\cong D[B]\cong \overleftrightarrow{K}_{(n+k-1)/2}.$$
Observe that $$\delta^{0}(D)= \lceil (n+k) / 2 \rceil - 2= (3k-1) / 2 - 1$$ and 
$$\lvert A \setminus B \rvert = \lvert B\setminus A \rvert = (3k-1) / 2 - (k-1) = (k+1) / 2.$$
Let $$A \setminus B = \{s_{1}, s_{2}, \dots, s_{(k+1)/2}\}, \ B \setminus A = \{t_{1}, t_{2}, \dots, t_{(k+1)/2}\}.$$ As now $k \geq 3$, we have $(k-1)/2\geq 1$. Choose~$(k-1)/2$ vertices from $A \cap B$ and denote them as $s_{(k+3)/2}, s_{(k+5)/2}, \dots, s_{k}$. The remaining $(k-1)/2$ vertices in $A \cap B$ are denoted as $t_{(k+3)/2}, t_{(k+5)/2}, \dots, t_{k}$. Let $$S=(A \setminus B) \cup \{s_{(k+3)/2}, s_{(k+5)/2}, \dots, s_{k}\}$$ and $$T=(B \setminus A) \cup \{t_{(k+3)/2}, t_{(k+5)/2}, \dots, t_{k}\}.$$ 

Suppose that there is an unpaired many-to-many $k$-DDPC, say $\mathbf{P}$, for $S$ and $T$ in $D$. Observe that $\mathbf{P}$ uses at most $(k-1)/2$ vertex-disjoint arcs from $A \setminus B$ to $\{t_{(k+3)/2}, t_{(k+5)/2}, \dots, t_{k}\}$ and at most $(k-1)/2$ vertex-disjoint arcs from $\{s_{(k+3)/2}, s_{(k+5)/2}, \dots, s_{k}\}$ to $B\setminus A$, that is, it uses at most $k-1$ vertex-disjoint arcs from $S$ to $T$, a contradiction.
Hence, $D$ is not unpaired many-to-many $k$-coverable.

\begin{case 6}
$k$ is even.
\end{case 6}

Let $D$ be a digraph of order $n = 2k$ such that
$$V(D)=A\cup B, |A\cap B|=k-2, D[A]\cong D[B]\cong \overleftrightarrow{K}_{(n+k)/2-1}.$$ Observe that
$$\delta^{0}(D)=\lceil (n+k) / 2 \rceil - 2= (3k) / 2 - 2$$
and 
$$\lvert A \setminus B \rvert = \lvert B\setminus A \rvert = (3k) / 2 - 1 - (k-2) = k/2 + 1.$$
Let 
$$A \setminus B = \{s_{1}, s_{2}, \dots, s_{k/2+1}\}, \ B \setminus A = \{t_{1}, t_{2}, \dots, t_{k/2+1}\}.$$ Choose $k/2-1$ vertices from $A \cap B$ and denote them as $s_{k/2+2}, s_{k/2+3}, \dots, s_{k}$, the remaining $k/2-1$ vertices in $A \cap B$ are denoted as $t_{k/2+2}, t_{k/2+3}, \dots, t_{k}$. Let 
$$S=(A \setminus B) \cup \{s_{k/2+2}, s_{k/2+3}, \dots, s_{k}\}, \ T=(B \setminus A) \cup \{t_{k/2+2}, t_{k/2+3}, \dots, t_{k}\}.$$ 

Suppose that there is an unpaired many-to-many $k$-DDPC, say $\mathbf{P}$, for $S$ and $T$ in $D$. Observe that $\mathbf{P}$ uses at most $k/2-1$ vertex-disjoint arcs from $A \setminus B$ to $\{t_{k/2+2}, t_{k/2+3}, \dots, t_{k}\}$ and at most $k/2-1$ vertex-disjoint arcs from $\{s_{k/2+2}, s_{k/2+3}, \dots, s_{k}\}$ to $B \setminus A$, that is, it uses at most $k-2$ vertex-disjoint arcs from $S$ to $T$, a contradiction. Hence,  $D$ is not unpaired many-to-many $k$-coverable. \qedhere

\end{proof}


\vspace{2mm}


Now we turn our attention to the paired many-to-many $k$-DDPC problem. In Theorem~\ref{The:2t2}, we will obtain a sharp Ore-type degree condition for the paired many-to-many $2$-DDPC problem. The example for the sharpness of the bound will be given in Proposition~\ref{Pro:2t2}.

\vspace{2mm}

\noindent
{\bf Theorem~\ref{The:2t2}.}
{\em  Let $D$ be a digraph of order $n$. If $d^{+}_{D}(x) + d^{-}_{D}(y) \geq n + 2$ for each $xy\not\in A(D)$, then $D$ is paired many-to-many $2$-coverable. Moreover, the bound for $d^{+}_{D}(x) + d^{-}_{D}(y)$ is sharp.}	


\begin{proof}
Let $S = \{s_{1}, s_{2}\}$ and $T = \{t_{1}, t_{2}\}$ be any disjoint source set and sink set of $V(D)$, respectively. Let $X = V(D)\setminus \{s_{2}, t_{1}\}$. We now construct a digraph $D_{1}$ of order $n_{1}(=n-1)$ by adding a new vertex $w$ such that
$$N^{+}_{D_{1}}(w) = N^{+}_{D}(s_{2}) \cap X, \ N^{-}_{D_{1}}(w) = N^{-}_{D}(t_{1}) \cap X.$$ 
For each $xy \notin A(D_{1})$, we have $$d^{+}_{D_{1}}(x) + d^{-}_{D_{1}}(y) \geq d^{+}_{D}(x) - 1 + d^{-}_{D}(y) - 1 \geq n + 2 - 2 = n = n_{1} + 1.$$ For each $x \in X \setminus N^{-}_{D_{1}}(w)$, we have $$d^{+}_{D_{1}}(x) + d^{-}_{D_{1}}(w) \geq d^{+}_{D}(x) - 1 + d^{-}_{D}(t_{1}) - 1 \geq n + 2 - 2 = n = n_{1} + 1.$$ Similarly, for each $y \in X \setminus N^{+}_{D_{1}}(w)$, we have $$d^{+}_{D_{1}}(w) + d^{-}_{D_{1}}(y) \geq d^{+}_{D}(s_{1}) - 1 + d^{-}_{D}(y) - 1 \geq n + 2 - 2 = n = n_{1} + 1.$$ 
Therefore, we conclude that $$d^{+}_{D_1}(x) + d^{-}_{D_1}(y) \geq n_1 + 1=|V(D_1)|+1$$ for each $xy\not\in A(D_1)$.

By Theorem~\ref{The:n+1}, $D_{1}$ is Hamiltonian connected, and so there is a Hamiltonian path, say $P'$, from $s_{1}$ to $t_{2}$ in $D_{1}$. Let $w^{+}$ (resp. $w^{-}$) denote the successor (resp. predecessor) of $w$ on $P'$. Let $$P_{1}:=s_{1}P'w^{-}t_{1}, P_{2}:=s_{2}w^{+}P't_{2}.$$ Observe that $\{P_{1}, P_{2}\}$ is a paired many-to-many $2$-DDPC for $S$ and $T$, that is, $D$ is paired many-to-many $2$-coverable. \qedhere
\end{proof}


\vspace{2mm}

\begin{proposition} \label{Pro:2t2}
    
There exists a digraph $D$ on $n \geq 9$ vertices with $d^{+}_{D}(x) + d^{-}_{D}(y)= n + 1$ for each $xy\not\in A(D)$, which is not paired many-to-many $2$-coverable. 
\end{proposition}
\begin{proof}
We define a digraph $D$ of order $n \geq 9$ as follows. Let $V(D) = A \cup B \cup S$ such that $$A \cap B = \{z\}, S = \{s_{1}, s_{2}, t_{1}, t_{2}\}, |A|= m \geq 3, |B|= n - m - 3 \geq 3.$$ Let the arc set $A(D)$ be defined as follows.
\begin{equation}
			\begin{split}
		A(D) &= \{uv, vu \colon u, v \in A\} \cup \{uv, vu \colon u, v \in B\} \\		
     &\cup (\{uv, vu \colon u, v \in S\} \setminus \{s_{1}t_{1}, s_{2}t_{2}\}) \\
				& \cup \{zu, uz \colon u \in S\}\\
				&\cup \{s_{1}u, us_{1}, t_{2}u, ut_{2}, t_{1}u, us_{2} \colon u \in A \setminus \{z\}\} \\
    &\cup \{t_{1}u, ut_{1}, s_{2}u, us_{2}, t_{2}u, us_{1} \colon u \in B \setminus \{z\}\}.
    \nonumber
			\end{split}
		\end{equation}



\begin{figure}[htb]
		\centering
		\begin{tikzpicture}
			\tikzset{arrow/.style = {draw = black, line width = 2pt, {Stealth[length = 2mm, width = 2mm]}-{Stealth[length = 2mm, width = 2mm]},}
			}
			\tikzset{arrow1/.style = {draw = black, line width = 2pt, -{Stealth[length = 3.5mm, width = 3mm]},}
			}
			\tikzset{arrow2/.style = {draw = black, line width = 2pt, {Stealth[length = 3.5mm, width = 3mm]}-{Stealth[length = 3.5mm, width = 3mm]},}
			}
			\tikzset{arrow3/.style = {draw = black, dashed, line width = 2pt, -{Stealth[length = 2.5mm, width = 2.5mm]},}
			}
			\draw[line width=1pt] (0, 6.5) rectangle (2, 0);
			\draw[line width=1pt] (7, 6.5) rectangle (9, 0);
			\draw[line width=1pt] (4.5, 5.8)   circle  [x radius=1.4cm, y radius=0.7cm, rotate=0];
			\draw[line width=1pt] (4.5, 2)   circle  [x radius=1cm, y radius=2.2cm, rotate=0];
			
			\filldraw[black]    (4.5, 5.8)      circle (2pt)   node [anchor=south west]  {$z$};
			\filldraw[black]    (4.5, 3.6)      circle (2pt)   node [anchor=south west]  {$s_{1}$};
			\filldraw[black]    (4.5, 2.6)      circle (2pt)   node [anchor=south east]  {$t_{1}$};	
			\filldraw[black]    (4.5, 1.6)      circle (2pt)   node [anchor=south]       {$s_{2}$};
			\filldraw[black]    (4.5, 0.6)      circle (2pt)   node [anchor=north west]  {$t_{2}$};
			
		    \draw[arrow] []      (2.1, 5.8)  -- (3, 5.8);
			\draw[arrow] []      (6, 5.8)  -- (6.9, 5.8);
			\draw[arrow] []      (4.5, 5)  -- (4.5, 4.3);
			
			\draw[arrow2] []      (2.1, 3.6)  -- (4.3, 3.6);
			\draw[arrow1] []     (4.3, 2.6)  -- (2.1, 2.6);
			\draw[arrow1] []     (2.1, 1.6)  -- (4.3, 1.6);
			\draw[arrow2] []      (2.1, 0.6)  -- (4.3, 0.6);
			
			\draw[arrow1] []     (6.9, 3.6)  -- (4.7, 3.6);
			\draw[arrow2] []      (4.7, 2.6)  -- (6.9, 2.6);
			\draw[arrow2] []      (4.7, 1.6)  -- (6.9, 1.6);
			\draw[arrow1] []     (4.7, 0.6)  -- (6.9, 0.6);
			
			\draw[arrow3] []     (4.5, 3.6)  -- (4.5, 2.6);
			\draw[arrow3] []     (4.5, 1.6)  -- (4.5, 0.6);
			
		\node at(1, 3.3)     {\large $A \setminus\{z\}$};
		\node at(8, 3.3)     {\large $B \setminus \{z\}$};
		\node at(5.3, 4.2)   {\large $S$};
		\node at(4.5, -1)    {\LARGE $D$};
		\end{tikzpicture}
		\caption{The graph of Proposition~\ref{Pro:2t2}.} 
		\label{fig:2t2}
	\end{figure}
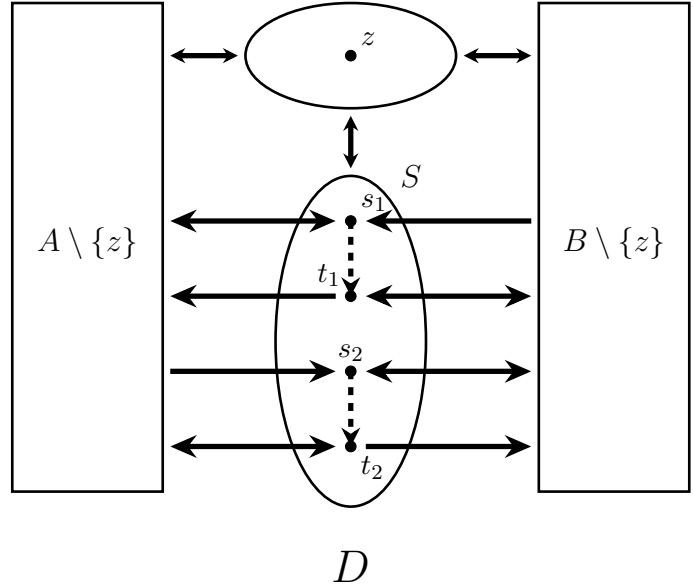

The graph is as shown in Figure~\ref{fig:2t2}, note that the dashed lines in this figure indicate that $s_{1}t_{1}, s_{2}t_{2} \notin A(D[S])$. It is not hard to check that for each $xy \notin A(D)$, we have $d^{+}_{D}(x) + d^{-}_{D}(y) = n + 1$. 
Also, observe that there do not exist disjoint $s_{1}$-$t_{1}$ path and $s_{2}$-$t_{2}$ path (otherwise, they must contain $z$, a contradiction). Hence, there is no paired many-to-many $2$-DDPC from $\{s_{1}, s_{2}\}$ to $\{t_{1}, t_{2}\}$ in $D$, that is, $D$ is not paired many-to-many $2$-coverable.
\qedhere
\end{proof}

\section{One-to-many $k$-DDPC problem}\label{sec:Proof3}








In the following result, we will make improvements for Theorem~\ref{The:otm} by replacing ``$n \geq Ck^{9}$" with ``$n \geq 3k$" and replacing ``$\delta^{0}(D) \geq \lceil (n + k + 1) / 2\rceil$" with ``$\delta^{0}(D) \geq \lceil (n + k)/2 \rceil$", which is sharp when $n+k$ is even and is sharp up to an additive constant 1 otherwise, by Proposition~\ref{pro4}.

\vspace{2mm}

\noindent
{\bf Theorem~\ref{The:main3}.}
{\em Let $D$ be a digraph of order $n \geq 3k$, where $k \geq 2$. If $\delta^{0}(D) \geq \lceil (n + k)/2 \rceil$, then $D$ is one-to-many $k$-coverable. Moreover, the bound for $\delta^{0}(D)$ is sharp when $n+k$ is even and is sharp up to an additive constant 1 otherwise.}
\begin{proof}

Let $S = \{s\}$ and $T = \{t_{1}, t_{2}, \dots, t_{k}\}$ be disjoint source set and sink set in $D$, respectively. 
As $$\delta^{0}(D) \geq \lceil (n + k)/2 \rceil \geq \lceil (3k + k)/2 \rceil=2k,$$ we have $d^{+}_{V(D) \setminus T}(s) \geq k$. Let $$S' = \{s_{1}, s_{2}, \dots , s_{k}\}, T = \{t_{1}, \dots, t_{k}\},$$ where $s_{1} = s$ and $\{s_{2}, s_{3}, \dots , s_{k}\}\subseteq N^{+}_{V(D) \setminus T}(s)$. Since $\delta^{0}(D) \geq \lceil (n + k)/2 \rceil$ and $n \geq 3k$, by Theorem~\ref{The:main1}, $D$ has an unpaired many-to-many $k$-DDPC for $S'$ and $T$, say $\{P_{1}, P_{2}, \dots , P_{k}\}$, where $P_{i}(1 \leq i \leq k)$ is an $s_{i}-t_{\sigma(i)}$ path and $\sigma$ is a permutation on $\{1, 2, \dots , k\}$. Recall that $ss_{i}\in A(D)$ for $2 \leq i \leq k$. Therefore, $\{P_{1}, P^{\ast}_{2}, \dots , P^{\ast}_{k}\}$ is a one-to-many $k$-DDPC for $S$ and $T$ in $D$, where $P^{\ast}_{i} := ss_{i}P_{i}t_{\sigma(i)}$ for $2 \leq i \leq k$. Hence, $D$ is one-to-many $k$-coverable. \qedhere
\end{proof}


\vspace{2mm}

\begin{proposition}\label{pro4}

For every $k\geq 2$ and every $n \geq k + 2$, there exists a digraph $D$ on $n$ vertices with minimum semi-degree $\lceil (n+k-1) / 2 \rceil - 1$ which is not one-to-many $k$-coverable. 
\end{proposition}

\begin{proof}
    
The proof is divided into the following two cases according to the parity of $n+k$. 
\begin{case 1}
$n+k$ is odd.
\end{case 1}
Let $D$ be a digraph of order $n \geq k + 2$ such that $$V(D)=A\cup B, |A\cap B|=k-1, D[A]\cong D[B]\cong \overleftrightarrow{K}_{(n+k-1)/2}.$$
Observe that $$\delta^{0}(D) =(n+k-1)/2-1=\lceil (n+k-1) / 2 \rceil - 1.$$ Let $$S=\{s\} \subseteq A \setminus B, T=\{t_{i} \colon 1 \leq i \leq k\} \subseteq B.$$ 
As $|A\cap B|=k-1$, there is no one-to-many $k$-DDPC for $S$ and $T$. Hence, $D$ is not one-to-many $k$-coverable. 
\begin{case 2}
	$n+k$ is even.
	\end{case 2}
Let $D_{1}\cong \overleftrightarrow{K}_{(n+k)/2-1}$ and $D_{2}\cong {E}_{(n-k)/2+1}$.
Let $D$ be a new digraph of order $n \geq k+2$ such that 	
$V(D) = V(D_{1}) \cup V(D_{2})$ and $$A(D) = A(D_{1}) \cup A(D_{2}) \cup \{xy, yx \colon x \in V(D_{1}), \, y \in V(D_{2})\}.$$
Observe that  $$\delta^{0}(D) =(n+k)/2-1 =\lceil (n+k-1) / 2 \rceil - 1.$$ 
Let $$S=\{s\} \subseteq V(D_{2}), T=\{t_{1}, t_{2}, \dots, t_{k}\} \subseteq V(D_{1}).$$ It can be checked that a set of disjoint paths $\{P_i \colon 1 \leq i \leq k\}$ covers at most $$((n+k)/2-1)-k+1=(n-k)/2(<|V(D_2)|)$$ vertices of $V(D_{2})$, where $P_i~(1 \leq i \leq k)$ is an $s-t_i$ path. Hence, $D$ is not one-to-many $k$-coverable.  \qedhere
\end{proof}

\section{One-to-one $k$-DDPC problem}

In the following result, we will improve  Theorem~\ref{The:oto} by replacing ``$n \geq Ck^{9}$" with ``$n \geq k+1$" and replacing ``$\delta^{0}(D) \geq \lceil (n + k + 1) / 2\rceil$" with ``$\delta^{0}(D) \geq \lceil (n + k-1)/2 \rceil$", which is sharp when $n+k$ is odd and is sharp up to an additive constant 1 otherwise, by Proposition~\ref{pro5}.

\vspace{2mm}

\noindent
{\bf Theorem~\ref{The:main4}.}
{\em  Let $D$ be a digraph of order $n \geq k + 1$, where $k \geq 2$. If $\delta^{0}(D) \geq \lceil (n + k - 1)/2 \rceil$, then $D$ is one-to-one $k$-coverable.	
Moreover, the bound for $\delta^{0}(D)$ is sharp when $n+k$ is odd and is sharp up to an additive constant 1 otherwise.}
\begin{proof}

We will prove the theorem by induction on $k$. For the base case that $k=2$, we prove the following claim:

\vspace{2mm}

\noindent{\bf Claim:} 
Let $D$ be a digraph of order $n \geq 3$. If $\delta^{0}(D) \geq \lceil (n + 1)/2 \rceil$, then $D$ is one-to-one $2$-coverable.

\vspace{2mm}

\noindent{\bf Proof of the claim:}
Let $S = \{s\}$ and $T = \{t\}$ such that $s$ and $t$ are distinct vertices of $D$. 
By Theorem~\ref{The:n+1}, $D$ is Hamiltonian connected as $\delta^{0}(D) \geq \lceil (n + 1)/2 \rceil$. Thus, there is a Hamiltonian path $P$ from $s$ to $t$. For each $v\in V(P)\setminus \{s\}$, we use $v^{-}$ to denote the predecessor of $v$ on $P$. 
Let $$X= \{v^{-} \colon v \in N^{+}_{D}(s)\}, \ Y = \{u \colon u \in N^{-}_{D}(t)\}.$$ As $t\not\in X \cup Y$, $\lvert X \cup Y \rvert \leq n-1$. Hence,
$$\lvert X \cap Y \rvert = \lvert X \rvert + \lvert Y \rvert - \lvert X \cup Y \rvert \geq \lceil (n+1)/2 \rceil + \lceil (n+1)/2 \rceil - (n-1) \geq 2.$$
We choose a vertex $w \in X \cap Y$, clearly $sw^{+}, wt \in A(D)$. Now $\{P_{1}, P_{2}\}$ is a one-to-one $2$-DDPC for $S$ and $T$, where $P_{1} = sw^{+}Pt$ and $P_{2} = sPwt$ (see Figure~\ref{fig:two}). 
Hence, $D$ is one-to-one $2$-coverable.\qed
\begin{figure}[htb]
			\centering
			\begin{tikzpicture}
\tikzset{arrow1/.style = {draw = black, thick, -{Latex[length = 3.5mm, width = 1.7mm]},}
			}
  \tikzset{arrow2/.style = {draw = black, thick, -{Latex[length = 2.7mm, width = 1.3mm]},}
			}              
    
				\filldraw[black]    (0, 0)  circle (3pt)  node [anchor=east] {$s$};
				\filldraw[black]    (3.5, 0)  circle (3pt)  node [anchor=south] {$w$};
				\filldraw[black]    (5, 0)  circle (3pt)  node [anchor=north] {$w^{+}$};
				\filldraw[black]    (8.5, 0)  circle (3pt)  node [anchor=west] {$t$};
				
				\draw[arrow1] [line width=2pt]      (0, 0) .. controls (2.5, 1.5) and (3.5, 1.5) .. (5, 0) ;
				\draw[arrow1] [line width=2pt]     (3.5, 0) .. controls (5, -1.5) and (6, -1.5) .. (8.5, 0) ;
				
				\draw[arrow1] [line width=2pt]     (0, 0) -- (3.5, 0);
				\draw[arrow2] []      (3.5, 0) -- (5, 0);
				\draw[arrow1] [line width=2pt]      (5, 0) -- (8.5, 0);
				
			\end{tikzpicture}
			\caption{The figure of the claim. The $s$-$t$ path is a Hamiltonian path. The thicker lines express two new $s$-$t$ paths: $P_{1} = sw^{+}Pt$ and $P_{2} = sPwt$. }
			\label{fig:two}
		\end{figure}
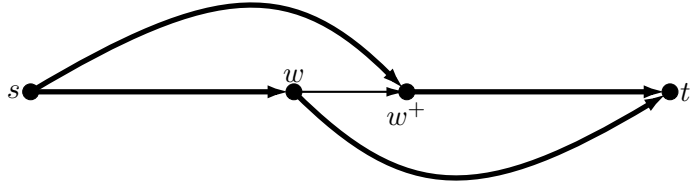

\vspace{2mm}

We next prove the inductive step. Assume that $k \geq 3$. Let $S^{\ast} = \{s^{\ast}\}$ and $T^{\ast} = \{t^{\ast}\}$ such that $s^{\ast}$ and $t^{\ast}$ are distinct vertices of $D$. As $\lvert N^{+}_{D}(s^{\ast}) \cup N^{-}_{D}(t^{\ast}) \rvert \leq n$, 
\begin{equation*}
\begin{split}
\lvert N^{+}_{D}(s^{\ast}) \cap N^{-}_{D}(t^{\ast}) \rvert 
& = \lvert N^{+}_{D}(s^{\ast}) \rvert + \lvert N^{-}_{D}(t^{\ast}) \rvert - \lvert N^{+}_{D}(s^{\ast}) \cup N^{-}_{D}(t^{\ast}) \rvert \\ & \geq \lceil (n+k-1)/2 \rceil + \lceil (n+k-1)/2 \rceil - n \geq 2.
\end{split}
\end{equation*}
We choose a vertex $h \in N^{+}_{D}(s^{\ast}) \cap N^{-}_{D}(t^{\ast})$. Let $D_{1} = D-\{h\}$ with order $n_{1} (= n-1\geq (k-1)+1)$. Observe that 
$$\delta^{0}(D_{1}) \geq \lceil (n+k-1)/2 \rceil - 1 = \lceil (n-1+k-1-1)/2 \rceil = \lceil (n_{1}+(k-1)-1)/2 \rceil.$$ By the induction hypothesis, $D_{1}$ is one-to-one $(k-1)$-coverable. Let $\{P_{1}, P_{2}, \dots , P_{k-1}\}$ be a one-to-one $(k-1)$-DDPC for $S^{\ast}$ and $T^{\ast}$ in $D_{1}$. Recall that $s^{\ast}h$, $ht^{\ast} \in A(D)$. Clearly, $\{P_{1}, P_{2}, \dots , P_{k-1}, P_{k}\}$ is a one-to-one $k$-DDPC for $S^{\ast}$ and $T^{\ast}$ in $D$, where $P_{k} := s^{\ast}ht^{\ast}$. Hence, $D$ is one-to-one $k$-coverable.\qedhere
\end{proof}


\vspace{2mm}

\begin{proposition}\label{pro5}
For every $k\geq 2$ and every $n \geq k + 1$, there exists a digraph $D$ on $n$ vertices with minimum semi-degree $\lceil (n+k) / 2 \rceil - 2$ which is not one-to-one $k$-coverable.
\end{proposition}

\begin{proof}
    
The proof is divided into the following two cases according to the parity of $n+k$. 

\begin{case 3}
$n+k$ is odd.
\end{case 3}
Let $D$ be a digraph of order $n \geq k+1$ such that $$V(D)=A\cup B, |A\cap B|=k-1, D[A]\cong D[B]\cong \overleftrightarrow{K}_{(n+k-1)/2}.$$
Observe that
$$\delta^{0}(D) = (n+k-1) / 2 - 1=\lceil (n+k) / 2 \rceil - 2.$$
Let 
$$S=\{s\} \subseteq A \setminus B, \ T=\{t\} \subseteq B \setminus A.$$ 
As $\lvert A \cap B \rvert = k-1$, there is no one-to-one $k$-DDPC for $S$ and $T$. Hence, $D$ is not one-to-one $k$-coverable.

\begin{case 4}
$n+k$ is even.
\end{case 4}
Let $D$ be a digraph of order $n \geq k+1$ such that 
$$V(D)=A\cup B, |A\cap B|=k-2, D[A]\cong D[B]\cong \overleftrightarrow{K}_{(n+k)/2-1}.$$
Observe that $$\delta^{0}(D) = (n+k) / 2 - 2=\lceil (n+k) / 2 \rceil - 2.$$
Let 
$$S=\{s\} \subseteq A \setminus B, \ T=\{t\} \subseteq B \setminus A$$ 
As $\lvert A \cap B \rvert = k-2$, there is no one-to-one $k$-DDPC for $S$ and $T$. Hence, $D$ is not one-to-one $k$-coverable.\qedhere
\end{proof}

\vskip 1cm

\noindent {\bf Acknowledgement.} We thank Professor Gexin Yu for his helpful suggestions and constructive comments. This work was supported by National Natural Science Foundation of China under Grant No. 12371352, Zhejiang Provincial Natural Science Foundation of China under Grant No. LY23A010011, Yongjiang Talent Introduction Programme of Ningbo under Grant No. 2021B-011-G, and Zhejiang Province College Students Science and
Technology Innovation Activity Plan and New Seedling Talent Program.

\end{sloppypar}

\begin{thebibliography}{20}
	\bibitem{Bang-Jensen09}J. Bang-Jensen and G. Gutin, Digraphs: Theory, Algorithms and Applications, 2nd Edition, Springer, London, 2009.
        \bibitem{Cao18}H. Cao, B. Zhang, Z. Zhou, one-to-one disjoint path covers in digraphs, Theoretical Computer Science, 714 (2018) 27-35.
        \bibitem{Chen10} X. Chen, Unpaired many-to-many vertex-disjoint path covers of a class of bipartite graphs, Information Processing Letters, 110 (2010) 203–205.
        \bibitem{Chud-Scott-SeymourAM} M. Chudnovsky, A. Scott and P.D. Seymour, Disjoint paths in tournaments, Advances in Mathematics, 270 (2015) 582--597.
        \bibitem{Chud-Scott-Seymour} M. Chudnovsky, A. Scott and P.D. Seymour, Disjoint paths in unions of tournaments, Journal of Combinatorial Theory, Series B, 135 (2019) 238--255.
        \bibitem{Ferrara} M. Ferrara, M. Jacobson and F. Pfender, Degree conditions for $H$-Linked digraphs, Combinatorics, Probability and Computing, 22 (2013) 684--699.
        \bibitem{Gould2006} R. Gould, A. Kostochka and G. Yu, On minimum degree implying that a graph is $H$-linked, SIAM Journal on Discrete Mathematics, 20 (2006) 829--840.
        \bibitem{Kuhn2008} D. K\"{u}hn, and D. Osthus, Linkedness and ordered cycles in digraphs, Combinatorics, Probability and Computing, 17 (2008) 689--709.
        \bibitem{Khn08}D. K\"{u}hn, D. Osthus, A. Young, $k$-ordered hamilton cycles in digraphs, Journal of Combinatorial Theory, Series B, 98 (6) (2008) 1165–1180.
        \bibitem{Li20}J. Li, C. Melekian, S. Zuo, E. Cheng, Unpaired many-to-many disjoint path covers on bipartite $k$-ary $n$-cube networks with faulty elements, International Journal of Foundations of Computer Science, 31 (3) (2020) 371-383.
        \bibitem{Lim16}H. S. Lim, H. C. Kim, J. H. Park, Ore-type degree conditions for disjoint path covers in simple graphs, Discrete Mathematics, 339 (2) (2016) 770–779.
        \bibitem{Liu-Rolek-Stephens-Ye-Yu}R. Liu, M. Rolek, C. Stephens, D. Ye and G. Yu, Connectivity for kite-linked graphs, SIAM Journal on Discrete Mathematics, 35(1) (2021), Article 431446.
        \bibitem{Lu21}H. L\"{u}, T. Wu, Unpaired many-to-many disjoint path cover of balanced hypercubes, International Journal of Foundations of Computer Science, 32 (2021) 943-956.
        \bibitem{MSZ23}A. Ma, Y. Sun, X. Zhang, A minimum semi-degree sufficient condition for one-to-many disjoint path covers in semicomplete digraphs, Discrete Mathematics, 346(8) (2023), Article 113401.
        \bibitem{Meng-Rolek-Wang-Yu}W. Meng, M. Rolek, Y. Wang and G. Yu, An improved linear connectivity bound for tournaments to be highly linked, European Journal of Combinatorics, 98 (2021), Article 103390.
        \bibitem{Ntafos}S.-C. Ntafos, S.-L. Hakimi, On path cover problems in digraphs and applications to program testing, IEEE Transactions on Software Engineering, 5(5) (1979) 520-529.
        \bibitem{Park16}J. H. Park, Unpaired many-to-many disjoint path covers in restricted hypercube-like graphs, Theoretical Computer Science, 617 (2016) 45-64.
        \bibitem{Park20}J.H. Park, H.S. Lim, Characterization of interval graphs that are unpaired 2-disjoint path coverable, Theoretical Computer Science, 821 (2020) 71-86.
        \bibitem{Park21}J. H. Park, A sufficient condition for the unpaired $k$-disjoint path coverability of interval graphs, The Journal of Supercomputing, 77 (2021) 6871-6888.
        \bibitem{Sun-Yeo} Y. Sun and A. Yeo, Directed Steiner tree packing and directed tree connectivity, Journal of Graph Theory, 102(1) (2023) 86--106.
        \bibitem{MOL1976}M. Overbeck-Larisch, Hamiltonian paths in oriented graphs, Journal of Combinatorial Theory, Series B, 21 (1976) 76-80.
        \bibitem{Thomas-Wollan}R. Thomas, P. Wollan, An improved linear edge bound for graph linkages, European Journal of Combinatorics, 26 (2005), 309--324
        \bibitem{Wang2019}F. Wang, W. Zhao, One-to-one disjoint path covers in hypercubes with faulty edges, The Journal of Supercomputing, 75(8) (2019) 5583-5595.
        \bibitem{Yu2018}G. Yu, Covering 2-connected 3-regular graphs with disjoint paths, Journal of Graph Theory, 88 (2018), 385-401.
\bibitem{Zhou-Qi-Yan}J. Zhou, Y. Qi and J. Yan, Improved results on linkage problems, Discrete Mathematics, 346(6) (2023), Article 113351.
 \bibitem{Zhou}Z. Zhou, Semi-degree conditions for one-to-many disjoint path covers in digraph, submitted.
       
\end{thebibliography}
\end{document}